\documentclass{article}

\usepackage{amsthm}
\usepackage{amsmath,amssymb,latexsym}
\usepackage[pdftex]{graphicx}
\usepackage{enumerate}
\theoremstyle{plain}
\newtheorem{thm}{Theorem}
\newtheorem*{xconj}{Conjecture}
\newtheorem{lem}{Lemma}
\newtheorem{prop}{Proposition}

\theoremstyle{remark}
\newtheorem*{rem}{Remark}
\numberwithin{equation}{section}
\begin{document}
\title{\textbf{On an inequality for the Riemann zeta-function in the critical strip}}
\author{Sadegh Nazardonyavi\footnote{sdnazdi@yahoo.com},\quad Semyon Yakubovich\footnote{syakubov@fc.up.pt}}
\date{}
\maketitle
\begin{center}
\emph{Department of Mathematics, Faculty of Sciences, University of Porto, Rua do Campo Alegre, 687\ \ 4169-007 Porto, Portugal} \footnote{Tel.: +351-220402130\qquad Fax: +351-220402108}
\end{center}

\begin{abstract}
By using new power inequalities we give an elementary proof of an important relation for the Riemann zeta-function $|\zeta(1-s)|\leq|\zeta(s)|$ in the strip $0<\Re s<1/2,\ |\Im s|\geq12$. Moreover, we establish a sufficient condition of the validity of the Riemann hypothesis in terms of the derivative with respect to $\Re s$ of $|\zeta(s)|^2$ and conjecture its necessity.
\end{abstract}
\section{Introduction and main result}
The Riemann zeta-function is defined as
\begin{equation}\label{zeta}
    \zeta(s)=\sum_{n=1}^\infty \frac{1}{n^s},\qquad(\Re s>1),
\end{equation}
and the series in (\ref{zeta}) converges absolutely. Let $s=\sigma+it$, where $\sigma$ and $t$ are real. The function $\zeta(s)$, defined by (\ref{zeta}) for $\sigma>1$, admits of analytic continuation over the whole complex plane having as its only singularity a simple pole with residue 1 at $s=1$ (\cite{Ivic},p. 1-3). The Riemann hypothesis (RH), stated by Riemann in 1859, concerns the complex zeros of the Riemann zeta function. The RH states that the non-real zeros of the Riemann zeta function $\zeta(s)$ all lie on the line at $\sigma= 1/2$ (\cite{Lag}).

During a study of the Riemann zeta-function, observing its graphs and looking for some relation between the RH and the size of the Riemann zeta function, an interesting problem arises to estimate its size in the critical strip; i.e. $|\zeta(1-s)|\leq|\zeta(s)|$ in the strip $0<\sigma<1/2,\ |t|\geq6.5$. To do this we employ a method of power inequalities related to some infinite product for $\pi$ and Euler's gamma-function instead of the use of Stirling's asymptotic formula (see \cite{Dixon, spira}). Namely, the main result of this Note is stated by the following
\begin{thm}\label{theorem 1}
Let $s=\sigma+it$, where $|t|\geq 12$. Then
\begin{equation}\label{main}
    |\zeta(1-s)|\leq|\zeta(s)|,\qquad \mbox{for}\quad 0<\sigma<\frac12,
\end{equation}
where the equality takes place only if $\zeta(s)=0$.\\
\end{thm}
\section{Auxiliary lemmas}
In order to prove this theorem, we will need some auxiliary elementary inequalities involving rational and logarithmic functions. Precisely, we have (see. \cite{Mitrin}, \S 2)
\begin{equation}\label{log(1+x)1}
    \frac{1}{x+1}<\log\left(1+\frac1x\right)<\frac1x,\qquad(x<-1,\ \mbox{or}\ \ x>0),
\end{equation}
\begin{equation}\label{log(1+x)2}
    \frac{1}{x+\frac12}<\log\left(1+\frac1x\right)<\frac1x,\qquad x>0,
\end{equation}
\begin{equation}\label{log(1+x)3}
    \frac{2x}{2+x}<\log(1+x)<\frac{x(2+x)}{2(1+x)},\qquad (x>0),
\end{equation}
\begin{equation}\label{log(1+x)4}
    \frac{x(2+x)}{2(1+x)}<\log(1+x)<\frac{2x}{2+x},\qquad (-1<x<0).
\end{equation}

Next we give some possibly new inequalities whose proofs are based on elementary calculus.

\begin{lem}
For any $t\geq1$
\begin{equation}\label{1+1/tx}
    \left(1+\frac{1}{tx+t-1}\right)^t\leq1+\frac{1}{x},\qquad(x\leq-1,\ x>0),
\end{equation}
\begin{equation}\label{1+x/t}
    \left(1+\frac{x}{t}\right)^t\leq1+\frac{2tx}{(1-t)x+2t},\qquad(0\leq x\leq2).
\end{equation}
Finally, for $0\leq a\leq1$
\begin{equation}\label{(1+1/x)a 1}
    \left(1+\frac{1}{x}\right)^{a}\geq1 + \frac{a}{x+1-a},\qquad(x\leq-1,\ x>0),
\end{equation}
where the equality holds only if $a=0,\ 1$ or $x=-1$, and
\begin{equation}\label{(1+1/x)a 2}
    \left(1+\frac{1}{x}\right)^{a}\geq1 + \frac{a}{x+\frac{1-a}{2}},\qquad(x>0),
\end{equation}
\begin{equation}\label{(1+1/x)a 3}
    \left(1+\frac{1}{x}\right)^{a}\leq1 + \frac{a}{x+\frac{1-a}{2}},\qquad(x\leq-1),
\end{equation}
where it becomes equality only if $a=0,\ 1$.\\
\end{lem}

\begin{proof}
In order to prove (\ref{1+1/tx}), we let
$$
f(t)=\left(1+\frac{1}{tx+t-1}\right)^t-(1+\frac{1}{x}),\qquad (x\leq-1,\ x>0).
$$
Then its derivative has the form
$$
f'(t)=\left(1+\frac{1}{tx+t-1}\right)^t \left(\log\left(1+\frac{1}{tx+t-1}\right)-\frac{1}{tx+t-1}\right).
$$
Calling inequality (\ref{log(1+x)1}), it is easily seen that $f'(t)<0$. Therefore $f(t)$ is decreasing and $f(1)=0$. Hence $f(t)<0$ for $t>1$. To verify (\ref{1+x/t}), observe that conditions $t\geq1$ and $0\leq x\leq2$ imply the positiveness of both sides of the inequality, which is equivalent to
$$
\left(1+\frac{x}{t}\right)^t\left(1+\frac{2tx}{(1-t)x+2t}\right)^{-1}\leq1,\qquad (0\leq x\leq2,\ \ t\geq1).
$$
Hence, denoting the left-hand side of the latter inequality by $g(x)$, we obtain
$$
g'(x)=\frac{\left(1+\frac{x}{t}\right)^t}{((1+t)x+2t)^2}(1-t^2)x^2\leq0,\qquad t\geq1.
$$
Since $g'(x)\leq0$, then $g(x)\leq g(0)=1$ for $x\geq0$. The equality in (\ref{1+x/t}) holds for $x=0$ or $t=1$. To prove (\ref{(1+1/x)a 1}), we replace $t=1/a$ in (\ref{1+1/tx}). The proof of (\ref{(1+1/x)a 2}) and (\ref{(1+1/x)a 3}) is straightforward and similar, invoking  with inequalities (\ref{log(1+x)3}) and (\ref{log(1+x)4}).
\end{proof}

\begin{lem} Let $0<\sigma <1/2$, $t\in\mathbb{R}$ and $x\geq(1+\sqrt{3})/4$. Then\\
$$
\hspace{-64mm}\frac{(2x+1-\sigma)^2+t^2}{(2x+\sigma)^2+t^2}<{\Bigg\{}\left(\frac{2x+1}{2x}\right)^2
$$
\begin{eqnarray}\label{2x+1-s}
\hspace{36mm}\times\left(1-\frac{(1+4x)((-1+\sigma)\sigma+t^2)}{(1+2x)^2((-1+\sigma)\sigma+t^2+4x^2)}\right){\Bigg\}}^{1-2\sigma}.
\end{eqnarray}
If $t\geq1/2$, it has
\begin{equation}\label{1-s/s}
    \frac{(1-\sigma)^2+t^2}{\sigma^2+t^2}<\left(1+\frac{1}{(-1+\sigma)\sigma+t^2}\right)^{1-2\sigma}.
\end{equation}
Finally, for $t\geq12$, the following inequality holds
\begin{equation}\label{1-s/s prod}
\left(\frac{(1-\sigma)^2+t^2}{\sigma^2+t^2}\right)\prod_{n=1}^3\frac{(2n+1-\sigma)^2+t^2}{(2n+\sigma)^2+t^2}<\left(\frac14\prod_{n=1}^3\left(\frac{2n+1}{2n}\right)^2\right)^{1-2\sigma}.
\end{equation}
\end{lem}

\begin{proof}
Let $1-2\sigma=1/y$. Then (\ref{2x+1-s}) is equivalent to
\begin{equation}\label{2x+1-s   y}
    \left(1+\frac{4(1+4x)}{y((-1/y+1+4x)^2+4t^2)}\right)^y<1+\frac{4(1+4x)y^2}{1+(-1+4t^2+16x^2)y^2}.
\end{equation}
It is not difficult to verify
\begin{equation}\label{0<1+4x<2}
    0<\frac{4(1+4x)}{(-1/y+1+4x)^2+4t^2}\leq2,\qquad (x\geq\frac{1+\sqrt{3}}{4},\ t\in\mathbb{R}).
\end{equation}
But relation (\ref{2x+1-s   y}) is just inequality (\ref{1+x/t}) where
$$
x:=\frac{4(1+4x)}{(-1/y+1+4x)^2+4t^2},\qquad t:=y.
$$
So we proved (\ref{2x+1-s}). In the same manner we establish (\ref{1-s/s}). To prove (\ref{1-s/s prod}) it is enough to verify the following inequality
$$
\left(1 +\frac{1}{(-1+\sigma)\sigma+t^2}\right)\prod_{n=1}^3\left(1-\frac{(1+4n)((-1+\sigma)\sigma+t^2)}{(1+2n)^2((-1+\sigma)\sigma+t^2+4n^2)}\right)<\frac14.
$$
Its left-hand side is increasing by $\sigma$ and decreasing by $t$ in the strip $]0,1/2[\times]1/2,\infty[$. Therefore, we may put $\sigma=1/2$ and $t=12$ and see by straightforward computation that it is less than $1/4$.
\end{proof}

\section{Proof of the main result}
\begin{proof}[Proof of Theorem \ref{theorem 1} ]
As it is known, the functional equation for the Riemann zeta-function (\cite{Titch}, p. 16) can be written as
$$
\pi^{-\frac12s}\Gamma\left(\frac12s\right)\zeta(s)=\pi^{-\frac12+\frac12s}\Gamma\left(\frac12-\frac12s\right)\zeta(1-s),
$$
or
$$
\zeta(1-s)=\pi^{\frac12-s}\frac{\Gamma(\frac12s)}{\Gamma(\frac12-\frac12s)}\zeta(s).
$$
Denoting by
$$
g(s)=\pi^{\frac12-s}\frac{\Gamma(\frac12s)}{\Gamma(\frac12-\frac12s)}
$$
we will show that for $0<\sigma<\frac12$ and $t\geq12$, $\ |g(\sigma+it)|<1$.

Taking the infinite product for the sine function (\cite{Ahlf}, p. 197)
$$
\sin \pi z=\pi z\prod_{n=1}^\infty(1-\frac{z^2}{n^2}),\qquad z\in\mathbb{C},
$$
and letting $z=\frac12$, we arrive at the known Wallis's formula
$$
\frac{\pi}{2}=\prod_{n=1}^\infty\frac{(2n)^2}{(2n-1)(2n+1)}.
$$
Moreover, the Gauss infinite product formula for the gamma function (\cite{Ref}, p. 61)
$$
\Gamma(z)=\frac1z\prod_{n=1}^\infty\displaystyle{\frac{\left(1+\frac1n\right)^z}{1+\frac{z}{n}}},
$$
yields
$$
\frac{\Gamma(\frac12s)}{\Gamma(\frac12-\frac12s)}=\frac{1-s}{s}\prod_{n=1}^\infty\left(\frac{1}{1+\frac1n}\right)^{\frac12-s}\left(\frac{1+\frac{1-s}{2n}}{1+\frac{s}{2n}}\right).
$$
Hence
\begin{eqnarray*}
  g(s) &=& \left(\frac{1-s}{s}\right)2^{\frac12-s}\prod_{n=1}^\infty\left(\frac{(2n)^2}{(2n-1)(2n+1)}\right)^{\frac12-s}\prod_{n=1}^\infty\left(\frac{1}{1+\frac1n}\right)^{\frac12-s}\left(\frac{1+\frac{1-s}{2n}}{1+\frac{s}{2n}}\right)\\
   &=& \left(\frac{1-s}{s}\right)2^{\frac12-s}\prod_{n=1}^\infty\left(\frac{(2n)^2n}{(2n-1)(2n+1)(n+1)}\right)^{\frac12-s}\prod_{n=1}^\infty\frac{1+\frac{1-s}{2n}}{1+\frac{s}{2n}}\\
   &=&\left(\frac{1-s}{s}\right)2^{\frac12-s}\prod_{n=1}^\infty\left(\frac{(2n)n}{(2n-1)(n+1)}\right)^{\frac12-s}\prod_{n=1}^\infty\left(\frac{2n}{2n+1}\right)^{\frac12-s}\left(\frac{1+\frac{1-s}{2n}}{1+\frac{s}{2n}}\right)\\
   &=&\left(\frac{1-s}{s}\right)2^{\frac12-s}\prod_{n=1}^\infty\left(\frac{(2n)n}{(2n-1)(n+1)}\right)^{\frac12-s}\prod_{n=1}^\infty\left(\frac{2n}{2n+1}\right)^{\frac12-s}\left(\frac{2n+1-s}{2n+s}\right)\\
   &=&\left(\frac{1-s}{s}\right)2^{\frac12-s}\prod_{n=1}^\infty\left(\frac{(2n+1)n}{(2n-1)(n+1)}\right)^{\frac12-s}\prod_{n=1}^\infty\left(\frac{2n}{2n+1}\right)^{1-2s}\left(\frac{2n+1-s}{2n+s}\right).
\end{eqnarray*}
We put
$$
f(s)=2^{\frac12-s}\prod_{n=1}^\infty\left(\frac{(2n+1)n}{(2n-1)(n+1)}\right)^{\frac12-s},
$$
and
$$
h(s)=h_1(s)h_2(s)
$$
where
$$
h_1(s)=\frac{1-s}{s},\qquad h_2(s)=\prod_{n=1}^\infty\left(\frac{2n}{2n+1}\right)^{1-2s}\frac{2n+1-s}{2n+s}.
$$
Plainly, for any $N$ we have
$$
\prod_{n=1}^N\frac{(2n+1)n}{(2n-1)(n+1)}=\frac{2N+1}{N+1}<2,
$$
and so
$$
\prod_{n=1}^\infty\frac{(2n+1)n}{(2n-1)(n+1)}=2.
$$
Hence
$$
|f(s)|=2^{1-2\sigma}.
$$
Therefore, it is sufficient to show that for $0<\sigma<\frac12$ and $t\geq12$
\begin{equation}\label{h(s)}
    |h(s)|<2^{2\sigma-1}.
\end{equation}

Indeed, $|h_1(s)|$ is a decreasing function with respect to $\sigma$ and $t$ for  $0<\sigma<1/2$ and $t>0$. Meanwhile
\begin{equation}\label{h2(s)}
    |h_2(s)|=\prod_{n=1}^\infty\left(\frac{2n}{2n+1}\right)^{1-2\sigma}\left|\frac{2n+1-s}{2n+s}\right|,
\end{equation}
is increasing with respect to $\sigma$ in the strip $(\sigma,t)\in ]0,1/2[\times [1/2,\infty[$, and decreasing with respect to $t$ in the strip $(\sigma,t)\in ]0,1/2[\times \mathbb{R}^+$.\\
Denoting by
$$
h_{2,n}(\sigma,t)=\left(\frac{2n}{2n+1}\right)^{1-2\sigma}\left|\frac{2n+1-(\sigma+it)}{2n+(\sigma+it)}\right|
$$
the general term of the product and assuming for now
\begin{equation}\label{h(2.n)}
    h_{2,n}(\sigma,t)<1,\qquad (0<\sigma<\frac12,\ t\geq0),
\end{equation}
we easily come out with the inequality
$$
\prod_{n=1}^{N+1} h_{2,n}(\sigma,t)<\prod_{n=1}^{N} h_{2,n}(\sigma,t),\qquad (0<\sigma<\frac12,\ t\geq0).
$$

To verify (\ref{h(2.n)}) we need to show that
\begin{equation}\label{1+1/2n}
    (1+\frac{1}{2n})^{1-2\sigma}>\sqrt{\frac{(2n+1-\sigma)^2+t^2}{(2n+\sigma)^2+t^2}},\qquad t\geq0.
\end{equation}
In fact,
\begin{equation}\label{extension t}
    \frac{(2n+1-\sigma)^2+t^2}{(2n+\sigma)^2+t^2}=1+\frac{(1-2\sigma)(4n+1)}{(2n+\sigma)^2+t^2}.
\end{equation}
Hence inequality (\ref{1+1/2n}) yields
\begin{equation}\label{1+1/2n 2}
    (1+\frac{1}{2n})^{1-2\sigma}>\frac{2n+1-\sigma}{2n+\sigma}\geq\sqrt{\frac{(2n+1-\sigma)^2+t^2}{(2n+\sigma)^2+t^2}}.
\end{equation}
However
$$
\frac{2n+1-\sigma}{2n+\sigma}=1+\frac{1-2\sigma}{2n+\sigma}.
$$
So the first inequality in (\ref{1+1/2n 2}) follows immediately from (\ref{(1+1/x)a 1}), letting $x=2n$ and $a=1-2\sigma$. Thus we have inequality (\ref{h(2.n)}).

Further, we show that $\{h_{2,n}(\sigma,t)\}_{n=1}^\infty$ is an increasing sequence for any $(\sigma,t)\in]0,1/2[\times\mathbb{R}$. To do this we consider the function $H_2(y)=h_{2,y}(\sigma,t)$ and differentiate it with respect to $y$. Hence by straightforward calculations we derive
$$
\hspace{-45mm}H_2'(y)=\frac{\displaystyle{\frac{1-2\sigma}{y(2y+1)}\left(\frac{2y}{2y+1}\right)^{1-2\sigma}}}{((2y+\sigma)^2+t^2)^2\ \sqrt{\displaystyle{\frac{(2y+1-\sigma)^2+t^2}{(2y+\sigma)^2+t^2}}}}
$$
$$
\hspace{-40mm}\times{\Big\{}(2y+1-\sigma)(1-\sigma)\sigma(2y+\sigma)
$$
$$
\hspace{-20mm}+(1+6y(1+2y)-2(1-\sigma)\sigma)t^2+t^4{\Big\}}.
$$
Since
$$
\hspace{-15mm}(2y+1-\sigma)(1-\sigma)\sigma(2y+\sigma)+(1+6y(1+2y)-2(1-\sigma)\sigma)t^2+t^4
$$
$$
\hspace{45mm}\geq(2y+1-\sigma)(1-\sigma)\sigma(2y+\sigma)>0,
$$
we get that the derivative is positive, and therefore $H_2(y)$ is increasing for $y>0$. Now fixing $t\geq1/2$ we justify that $h_{2,n}(\sigma,t)$ is increasing by $\sigma$. Precisely,
$$
\hspace{-35mm}\frac{\partial}{\partial\sigma}h_{2,n}(\sigma,t) = \left(\frac{2n}{2n+1}\right)^{1-2\sigma}/\left|\frac{2n+1-(\sigma+it)}{2n+(\sigma+it)}\right|
$$
$$
\hspace{-9mm}\times{\Big\{}-(1+4n)(4n^2+2n+\sigma-\sigma^2+t^2)
$$
$$
\hspace{23mm}+2((2n+1-\sigma)^2+t^2)((2n+\sigma)^2+t^2)\log(1+\frac{1}{2n}){\Big\}}
$$
and we achieve the goal showing that the latter multiplier is positive. But this is true due to inequality (\ref{log(1+x)2}), because it is greater than
$$
\frac{-(1-2\sigma)^2(2n+1-\sigma)(2n+\sigma)+(8n(1+2n)+3-8(1-\sigma)\sigma)t^2+4t^4}{1+4n} \\
$$
$$
\geq\frac{1+(1-\sigma)\sigma(8n(1+2n)-3+4(1-\sigma)\sigma)}{1+4n}>0,\qquad(0<\sigma<1/2,\ t\geq1/2).
$$
Returning to  (\ref{h2(s)}) we conclude that $|h_2(\sigma,t)|$ is increasing with respect to $\sigma$ for $0<\sigma<\frac12$ and $t\geq1/2$, and by (\ref{extension t}) it is decreasing with respect to $t$ for $0<\sigma<\frac12$ and $t>0$.\\

Since
\begin{equation}\label{hN(s)}
    |h_N(s)|=|\frac{1-s}{s}|\prod_{n=1}^N\left(\frac{2n}{2n+1}\right)^{1-2\sigma}\left|\frac{2n+1-s}{2n+s}\right|
\end{equation}
is decreasing by $N$ we have
$$
|h(s)|\leq|h_N(s)|.
$$
As $|h_N(s)|$ is decreasing by $t$, it is enough to show that
$$
|h_N(s)|<2^{2\sigma-1},\qquad\mbox{for}\qquad(t=12,\ N=3)
$$
and this has been established in (\ref{1-s/s prod}). Moreover, since $\zeta(s)$ is reflexive with respect to the real axis, i.e., $\zeta(\overline{s})=\overline{\zeta(s)}$, inequality (\ref{main}) holds also for $t\leq-12$. Theorem \ref{theorem 1} is proved.
\end{proof}
\begin{rem}
A computer simulation shows that the main result is still valid for $t\in ]6.5, 12[$ (See Figure \ref{b}). However, a direct proof by this approach is more complicated, because to achieve the goal we should increase a number $N$ of terms in the product (\ref{hN(s)}).
\end{rem}
\begin{figure}[hb]
  \centering
  \includegraphics[scale=.20]{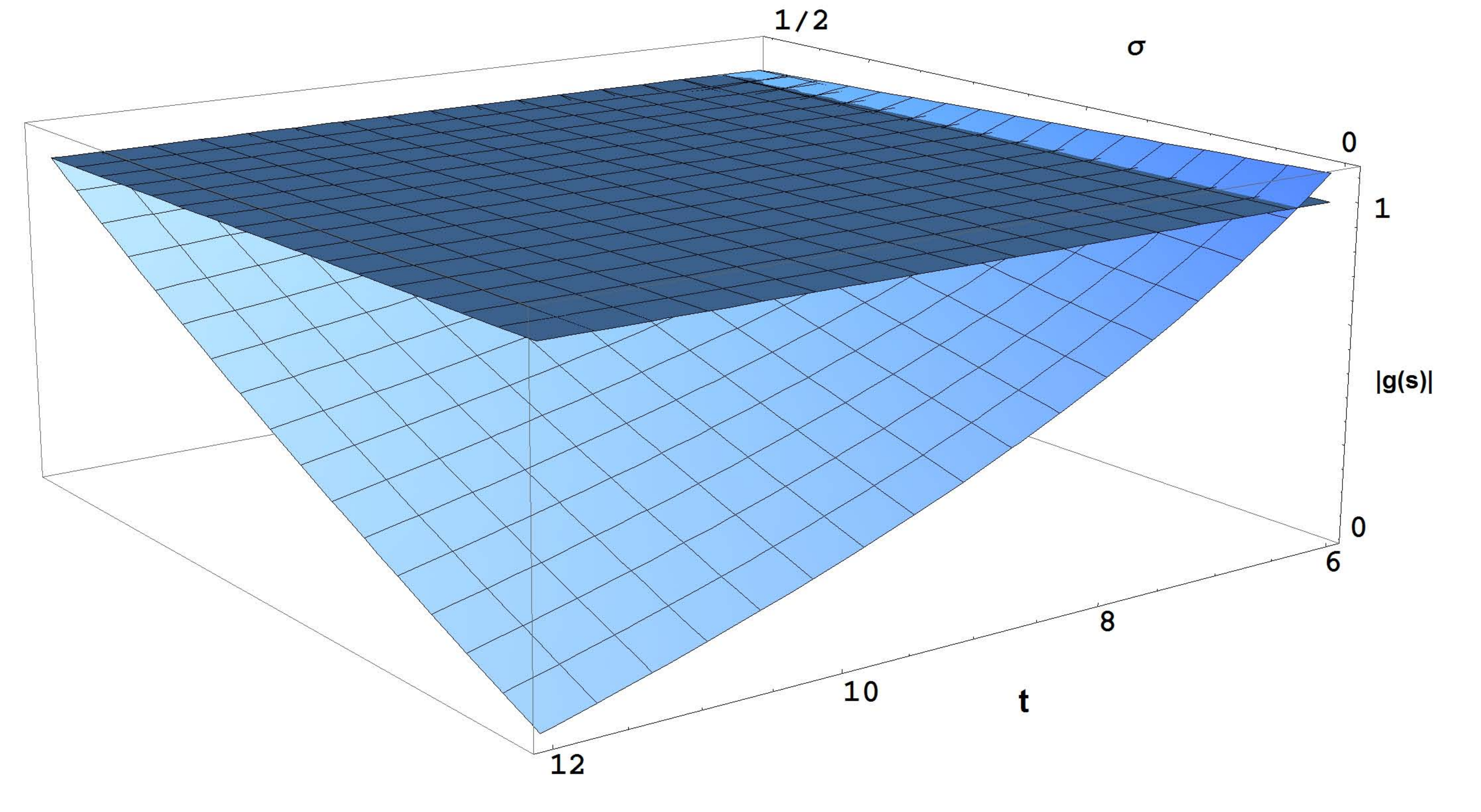}\\
  \caption{The graph of $|g(s)|$ for $6<t<12$}\label{b}
\end{figure}
\section{Conclusion and some result}
Here as in \cite{spira} one can announce the following proposition.
\begin{prop}
The Riemann hypothesis is true if and only if
$$
|\zeta(1-s)|<|\zeta(s)|,\qquad \mbox{for}\quad (0<\sigma<\frac12,\ |t|>6.5).
$$
\end{prop}

As it is known \cite{Duen},  zeros of the derivative $\zeta'(s)$ of Riemann's zeta-function are connected with the behavior of zeros of $\zeta(s)$ itself. Indeed, Speiser's theorem \cite{speiser} states that the Riemann hypothesis (RH) is equivalent to $\zeta'(s)$ having no zeros on the left of the critical line. Thus, we can get further tools to study RH, employing these properties.

Finally, we will formulate a sufficient condition for the Riemann hypothesis to be true.
\begin{prop}
If
$$
\qquad\frac{\partial}{\partial\sigma}|\zeta(s)|^2<0,\qquad \mbox{for}\quad (0<\sigma<\frac12,\ |t|>6.5),\eqno(A)
$$
then the Riemann hypothesis is true.
\end{prop}
\begin{proof}
In fact, if the Riemann hypothesis were not true,  then by Speiser's theorem \cite{speiser}, there exists a number $s\in ]0,1/2[\times\mathbb{R}$, such that $\zeta'(s)=0$. Hence $\frac{\partial}{\partial\sigma}|\zeta(s)|^2=0$.\\
\end{proof}

We conclude this paper by the following
\begin{xconj}
The condition (A) is also necessary for the validity of the Riemann hypothesis.
\end{xconj}

\paragraph{ACKNOWLEDGMENTS.} The work of the first author is supported by the Calouste Gulbenkian Foundation, under Ph.D. grant number CB/C02/2009/32. Research partially funded by the European Regional Development Fund through the programme COMPETE and by the Portuguese Government through the FCT
 under the project PEst-C/MAT/UI0144/2011.

\bibliographystyle{plain}

\end{document}